\PassOptionsToPackage{unicode}{hyperref}
\PassOptionsToPackage{hyphens}{url}
\documentclass[
]{article}
\usepackage{amsmath,amssymb,amsthm}
\usepackage{lmodern}
\usepackage{iftex}
\ifPDFTeX
  \usepackage[T1]{fontenc}
  \usepackage[utf8]{inputenc}
  \usepackage{textcomp} 
\else 
  \usepackage{unicode-math}
  \defaultfontfeatures{Scale=MatchLowercase}
  \defaultfontfeatures[\rmfamily]{Ligatures=TeX,Scale=1}
\fi
\IfFileExists{upquote.sty}{\usepackage{upquote}}{}
\IfFileExists{microtype.sty}{
  \usepackage[]{microtype}
  \UseMicrotypeSet[protrusion]{basicmath} 
}{}
\makeatletter
\@ifundefined{KOMAClassName}{
  \IfFileExists{parskip.sty}{%
    \usepackage{parskip}
  }{
    \setlength{\parindent}{0pt}
    \setlength{\parskip}{6pt plus 2pt minus 1pt}}
}{
  \KOMAoptions{parskip=half}}
\makeatother
\usepackage{xcolor}
\IfFileExists{xurl.sty}{\usepackage{xurl}}{} 
\IfFileExists{bookmark.sty}{\usepackage{bookmark}}{\usepackage{hyperref}}
\hypersetup{
  hidelinks,
  pdfcreator={LaTeX via pandoc}}
\providecommand{\tightlist}{%
  \setlength{\itemsep}{0pt}\setlength{\parskip}{0pt}}
\ifLuaTeX
  \usepackage{selnolig}  
\fi

\usepackage[numbers]{natbib}

\theoremstyle{theorem}
\newtheorem{theorem}{Theorem}[section]

\theoremstyle{lemma}
\newtheorem{lemma}{Lemma}[section]

\theoremstyle{definition}
\newtheorem{definition}{Definition}[section]

\theoremstyle{proposition}
\newtheorem{proposition}{Proposition}[section]

\theoremstyle{corollary}
\newtheorem{corollary}{Corollary}[section]

\theoremstyle{remark}
\newtheorem{remark}{Remark}[section]

\theoremstyle{conjecture}
\newtheorem{conjecture}{Conjecture}[section]

\theoremstyle{definition}
\newtheorem{example}{Example}[section]

\author{Ran Gutin \\ Department of Computer Science, Imperial College London}
\date{\today}
\title{Unitary canonical forms over Clifford algebras, and an observed unification of some real-matrix decompositions}

\newcommand{\cl}{Cl}
\newcommand{\unpack}{\operatorname{unpack}}

\begin{document}

\maketitle

\begin{abstract}
  We show that the spectral theorem -- which we understand to be a statement that every self-adjoint matrix admits a certain type of canonical form under unitary similarity -- admits analogues over other $*$-algebras distinct from the complex numbers. If these $*$-algebras contain nilpotents, then it is shown that there is a consistent way in which many classic matrix decompositions -- such as the Singular Value Decomposition, the Takagi decomposition, the skew-Takagi decomposition, and the Jordan decomposition, among others -- are immediate consequences of these. If producing the relevant canonical form of a self-adjoint matrix were a subroutine in some programming language, then the corresponding classic matrix decomposition would be a 1-line invocation with no additional steps. We also suggest that by employing operator overloading in a programming language, a numerical algorithm for computing a unitary diagonalisation of a complex self-adjoint matrix would generalise immediately to solving problems like SVD or Takagi. While algebras without nilpotents (like the quaternions) allow for similar unifying behaviour, the classic matrix decompositions which they unify are never obtained as easily. In the process of doing this, we develop some spectral theory over Clifford algebras of the form $\cl_{p,q,0}(\mathbb R)$ and $\cl_{p,q,1}(\mathbb R)$ where the former is admittedly quite easy. We propose a broad conjecture about spectral theorems.
\end{abstract}

\section{Introduction}

In this paper, we prove ``spectral theorems'' for all Clifford $*$-algebras with a limited amount of degeneracy (up to 1 nilpotent generator) and suggest that a ``unification'' between some classic matrix decompositions results from this. The unification is interesting because: If producing the canonical form were a subroutine, then these classical matrix decompositions would be obtained from 1 application of the subroutine, and nothing more. Other reductions between matrix decompositions are usually less ``efficient'' (in the sense of not being just 1 application with no additional steps). Some connections with the theory of quiver representations are also obtained.

The classic matrix decompositions we consider are the:
\begin{itemize}
    \item The unitary diagonalisation of a self-adjoint matrix.
    \item The Takagi decomposition \cite{autonne_sur_1902} of a complex-symmetric matrix.
    \item The skew-Takagi decomposition \cite{teretenkov_singular_2021} of a complex skew-symmetric matrix.
    \item The SVD of a real matrix.
    \item The Jordan decomposition of a real matrix.
\end{itemize}
In some sense, the first of these cases is equivalent to the rest, if one varies the $*$-algebra. We will now discuss the notion of a $*$-algebra, and its motivation.

We are intending to generalise notions like \emph{self-adjoint matrix}, \emph{unitary matrix} and \emph{singular value decomposition} (among others) to various ``number systems'' for various reasons. These notions were originally developed over the complex numbers $\mathbb C$. If we focus attention to the notion of, let's say, a unitary matrix, we see that to generalise this to novel ``number systems'' it is not sufficient to simply redefine the operations $\{+,-,\times,\div\}$, but also the \emph{complex conjugation} operation, which we will denote $*$. A bit of experience with similar ``number systems'' (like the quaternions, or the $2 \times 2$ real matrices) suggests that a promising generalisation of complex conjugation over a ring would be an arbitrary ring anti-automorphism of order up to $2$. Such an operation is called an \emph{involution}, and we must include it in our list of operations to redefine: $\{+,-,\times,\div,*\}$.

The above discussion suggests that the claim in some linear algebra courses that linear algebra takes place over \emph{fields} is incorrect, because notions like unitary matrices are defined in terms of involutions which are not field operations. It is thus interesting to suggest that linear algebra might be done instead over $*$-fields, which are fields equipped with involutions. Unfortunately, this is not sufficient for our paper, where our ``number systems'' may contain zero divisors.

When we refer to an algebra over a field, we understand this to be something unital, associative and finite-dimensional. The notion of an algebra is not sufficient because there are many involutions which an algebra can be equipped with. The notion of a $*$-algebra \cite{nlab_authors_star-algebra_2022} is clearly a better notion of ``number system'' than just an algebra when generalising spectral theory. Our $*$-algebras will be over $*$-fields in the expected way.

Most of the $*$-algebras we'll consider will be Clifford $*$-algebras over a $*$-field $\mathbb F$ which will be either the real number $\mathbb R$ or the complex numbers $\mathbb C$ equipped with their standard involutions. We will in fact consider \emph{two} different involutions for each algebra.

To give a very quick example of why passing from $\mathbb C$ (equipped with its standard involution) to a larger $*$-algebra can unify matrix decompositions, consider the Takagi decomposition (though the SVD would provide a very similar example): This states that given a $\mathbb C$-matrix $M$ which satisfies $M = M^T$, there is a unitary matrix $U$ and a $\mathbb R$-diagonal matrix $D$ such that $M = UDU^T$. Notice that while this \emph{looks like} a diagonalisation of a linear map, it actually isn't one because $U^T \neq U^{-1}$. The columns of $U$ are not eigenvectors. But now imagine introducing a new imaginary number $\delta$ which satisfies $\delta i = -i \delta$ and $\delta^2 = 0$. We then have that $M \delta = UDU^T \delta = U (D \delta) U^* = U(D \delta) U^{-1}$. We see that while the columns of $U$ are not eigenvectors of $M$, they are instead eigenvectors of $M \delta$. Additionally, if we extend the involution $*$ so that $\delta^* = \delta$, we get that $M \delta$ is self-adjoint, which $M$ wasn't. If we brazenly assume that $M \delta$ can be unitarily diagonalised (by an optimistic extension of the spectral theorem), then it's immediate that the unitary diagonalisation will take the form $M \delta = U (D \delta) U^*$ where $U$ is immediately a $\mathbb C$-matrix and $D$ is immediately an $\mathbb R$-matrix. This is more efficient than introducing the quaternion $j$ which satisfies $ji = ij$ but not $j^2 = 0$, because we cannot immediately conclude that the unitary diagonalisation of $Mj$ will yield the Takagi decomposition.

\subsection{Preliminary definitions}

We assume the reader knows what a field is. A $*$-field is a pair $(\mathbb F,*_{\mathbb F})$ where $\mathbb F$ is a field and $*_{\mathbb F} : \mathbb F \to \mathbb F$ is a function called the involution. The involution satisfies $(x + y)^* = x^* + y^*$, $(xy)^* = x^* y^*$, $(-x)^* = -x^*$, $1^* = 1$, $0^* = 0$.

A $*$-algebra $(\mathcal A, * : \mathcal A \to \mathcal A)$ over a $*$-field $(\mathbb F, *_{\mathbb F})$ is an algebra over $\mathbb F$ equipped with a map $*:\mathcal A \to \mathcal A$ which we call the involution, satisfying $(x^*)^* = x$, $(xy)^* = y^* x^*$ and $(x + \lambda y)^* = x^* + \lambda^{*_{\mathbb F}} y^*$. In this paper, we assume that our $*$-algebras are both associative and unital. Associativity means $(xy)z = x(yz)$. Unital means that there exists an element $1 \in \mathcal A$ such that $x1 = 1x = x$, and $1^* = 1$.

We hope that it is clear what a matrix should be over a $*$-algebra $(\mathcal A, * : \mathcal A \to \mathcal A)$. We define a map which we \emph{also call $*$} over matrices over a $*$-algebra, which we define to satisfy $(M^*)_{ij} = (M_{ji})^*$ (where the $*$ on the right hand side is the involution over the $*$-algebra, but the $*$ on the LHS is a map over matrices), which we call either the \emph{adjoint} or the \emph{conjugate-transpose}. A self-adjoint matrix is one which satisfies $M^* = M$, and a unitary matrix is one which satisfies $M^* = M^{-1}$.

\section{The double number example, and Jordan decomposition} \label{double-number-examples}

We will illustrate how matrix-decomposition-unification and ``spectral theorems'' over ``exotic'' $*$-algebras relate to each other by considering the Clifford algebra $Cl_{1,0,0}(\mathbb R)$ (which is merely an algebra until equipped with an involution). This algebra is called by different names: Sometimes it's called the ``double numbers'' or the ``split-complex numbers''. We will call it the double numbers \cite{gutin_matrix_2021}. This algebra is sometimes defined as being ``like the complex numbers'' but with the role of $i$ being replaced with a number $j$ for which $j^2=1$. This results in the number $1$ having 4 different square roots. This algebra is isomorphic to $\mathbb R \oplus \mathbb R$. This implies that the algebra is far more convenient to work with when its elements are written as $(a,b)$, and all arithmetic operations $\{+,-,\times,\div\}$ are understood to happen componentwise. We now consider two possible involutions.

\subsection{The sterile $*_1$ involution}

The first involution will be denoted $*_1$. This will be defined by $(a,b)^{*_1} = (a,b)$. The corresponding $*$-algebra over $\mathbb R$ will be denoted \newcommand{\rrtriv}{(\mathbb R \oplus \mathbb R, *_1)}$\rrtriv$. This is a trivial definition, and will result in a rather sterile ``spectral theorem''. How should we write an $\rrtriv$-matrix? We will write it in the form $(M,K)$ where $M$ and $K$ are real matrices of equal dimensions. What we would be the adjoint operation (sometimes called conjugate-transpose)? It would simply be $(M,K)^* = (M^T,K^T)$. Based on this, we see that $(M,K)$ is unitary (respectively, self-adjoint) whenever $M$ and $K$ are individually unitary (respectively, self-adjoint). Trivially we obtain a spectral theorem: Given a self-adjoint $\rrtriv$-matrix $(H,K)$, we see that there exists an $\rrtriv$-unitary matrix $(U,V)$ and an $\rrtriv$-diagonal matrix $(D,E)$ such that $(H,K) = (U,V) (D,E) (U,V)^*$. But this is trivial and uninteresting.

\subsection{The intriguing $*_{-1}$ involution}

More interesting would be to consider the other possible involution. This involution will be denoted $*_{-1}$, but otherwise just $*$ if ambiguity won't arise. This will be defined by $(a,b)^{*_{-1}} = (b,a)$. The corresponding $*$-algebra over $\mathbb R$ will be denoted \newcommand{\rrnon}{(\mathbb R \oplus \mathbb R, *_{-1})}$\rrnon$. We argue that an $\rrnon$-matrix should now be written as $[M,K]$, and define this to mean $(1,0)M + (0,1)K^T$. We use square brackets instead of round brackets because of the transpose, which ultimately serves to simplify things.\footnote{We think it would be fairer to reserve round bracket notation $(M,K)$ for when it stands for $(1,0)M + (0,1)K$.} We observe the following identities:
$$\begin{aligned}[A,B][C,D] &= [AC,DB] \\ [A,B] + [C,D] &= [A+C,B+D] \\ [A,B]^* &= [B,A]\end{aligned}$$
By defining the square bracket notation the way we did, we made multiplication slightly more complicated while greatly simplifying the adjoint operation (which is the third one in our list). The multiplication operation is made more complicated because the second component multiplies in the opposite order to the first component.

What is now a self-adjoint matrix over $\rrnon$? It needs to satisfy $[M,K]^* = [M,K]$. This simplifies to $[K,M] = [M,K]$, so it is precisely of the form $[M,M]$ where $M$ is an arbitrary square real matrix. Note that the embedding $M \mapsto [M,M]$ of square real matrices into $\rrnon$-matrices does not preserve multiplication, and so is not a ring homomorphism.

What is now an $\rrnon$-unitary matrix? It needs to satisfy $[M,K]^* = [M,K]^{-1}$. This simplifies to $[K,M] = [M^{-1},K^{-1}]$, so it is precisely of the form $[P,P^{-1}]$ where $P$ is an arbitrary invertible real matrix.

Given a self-adjoint $\rrnon$-matrix $[M,M]$, the matrices unitarily similar to it are all of the form $[P,P^{-1}] [M,M] [P,P^{-1}]^*$, which simplifies to $[PMP^{-1}, PMP^{-1}]$. The ``spectral theorem'' for $\rrnon$ is thus essentially equivalent to the Jordan decomposition for $\mathbb R$-matrices: $[M,M] = [P,P^{-1}][J,J][P^{-1},P]$ where $J$ is a unique Jordan matrix. Thus, the spectral theorem for $\rrnon$ is (somehow) the same as the $\mathbb R$-Jordan decomposition. This achieves a unification.

In summary:

\newcommand{\doublenontriv}{(\mathbb R \oplus \mathbb R, *_{-1})}
\begin{definition}
    The $*$-algebra $\doublenontriv$ has as its involution $(a,b)^{*_{-1}} = (b,a)$. A matrix over $\doublenontriv$ is written in the form $[M,K]$, defined to mean $(1,0)M + (0,1)K^T$. The square brackets hint at the transpose.
\end{definition}
\begin{proposition}
    A self-adjoint $\doublenontriv$-matrix is of the form $[M,M]$ where $M$ is an arbitrary square $\mathbb R$-matrix. A unitary $\doublenontriv$-matrix is of the form $[P,P^{-1}]$ where $P$ is an invertible $\mathbb R$-matrix.
\end{proposition}
\begin{remark}
    The mapping $M \mapsto [M,M]$ preserves the additive, but not the multiplicative, structure of $\mathbb R$-matrices.
\end{remark}
\begin{theorem} Every self-adjoint $\doublenontriv$-matrix is unitarily similar to a unique matrix of the form $[J,J]$, where $J$ is a Jordan matrix, and the uniqueness is up to permutation of the Jordan blocks of $J$.
\end{theorem}

The proofs are immediate.

\section{What we propose a spectrum theorem is for general $*$-algebras, and a conjecture}

In this section, we define what we think a spectral theorem ought to be, and pose a conjecture. A proof of this would be like a (sometimes purely qualitative and non-constructive) generalisation of many matrix decompositions. We begin by defining terms:

We recall some definitions related to monoids:

\begin{definition}
  A \emph{monoid} is the same notion as a group, but without the requirement of existence of inverses. An \emph{abelian monoid} is a monoid where the product is commutative. The product in an abelian monoid is usually written using the additive symbol $+$. A \emph{free abelian monoid generated by a set $S$} is the monoid whose underlying set is the set of functions of the form $f:S \to \mathbb N$ with finite support (i.e. where the set $\{x \in S \mid f(x) \neq 0\}$ is finite), where the monoid operation is $(f+g)(x) = f(x) + g(x)$. We define a \emph{subfree abelian monoid} to be a submonoid of a free abelian monoid.
\end{definition}

Let $(\mathcal A,*)$ be a $*$-algebra over $\mathbb R$.

\begin{definition}
  By the \emph{spectral monoid} of $(\mathcal A,*)$, which we denote $\overline{\operatorname{Her}}(\mathcal A,*)$, we mean the monoid ${\operatorname{Her}}(\mathcal A,*)/{\sim}$ where ${\operatorname{Her}}(\mathcal A,*)$ is the monoid formed from the set of Hermitian matrices over $(\mathcal A,*)$ with the monoid operation being direct sum of matrices $\oplus$, and $\sim$ denotes the equivalence relation \emph{unitary similarity}: $M \sim K \iff \exists\, \text{unitary matrix} \,U: UMU^* = K$.
\end{definition}

\begin{conjecture}[The spectral conjecture] \label{conjecture}
    Let $(\mathcal A,*)$ be an arbitrary associative, unital, finite-dimensional $*$-algebra over $\mathbb R$. We conjecture that the spectral monoid $\overline{\operatorname{Her}}(\mathcal A,*)$ is a subfree abelian monoid.
\end{conjecture}

\begin{example}
  We refer to only the examples in section \ref{double-number-examples}. For $(\mathcal A,*) = \doublenontriv$, $\overline{\operatorname{Her}}(\mathcal A,*)$ is a free abelian monoid with the generators corresponding to real Jordan blocks. This captures the uniqueness of the Jordan Normal Form of an $\mathbb R$-matrix, up to permutation of the Jordan blocks. For $(\mathcal A,*) = \rrtriv$, $\overline{\operatorname{Her}}(\mathcal A,*)$ is only a subfree abelian monoid; it is in fact a submonoid of $\overline{\operatorname{Her}}(\mathbb R,\operatorname{id}_{\mathbb R})$.
\end{example}

\begin{remark}
  We will later show that the Singular Value Decomposition for the $*$-algebra $(\mathcal A, *)$ is implied by the spectral theorem for the $*$-algebra $(\mathcal B, \dagger)$ (called the ``SVD algebra'' for $(\mathcal A, *)$) where $\mathcal B$ is the result of adjoining an element $\delta$ to $\mathcal A \oplus \mathcal A$ (where $\oplus$ denotes direct sum of algebras), such that a general element of $\mathcal B$ is of the form $(x,y) + (y',x')\delta$, with: $\delta^2=0$, $\delta(x',y') = (y',x')\delta$, $\delta^\dagger = \delta$, $(x,y)^\dagger = (x^*,y^*)$. The spectral theorem for the enlarged $*$-algebra is more general than the SVD for the original $*$-algebra, but nevertheless provides insight into its SVD.
\end{remark}

\section{Introducing three $*$-algebras corresponding to the SVD, Takagi and skew-Takagi decompositions respectively}

\subsection{The ``SVD $*$-algebra''}

Consider the Clifford algebra $\cl_{1,0,1}(\mathbb R)$, which we will equip with a certain involution. To make the reader's life easier, we will describe this algebra explicitly. The elements of this algebra are all of the form
$$(a,b) + (a',b')\delta$$
where $a,b,a',b'$ are all real numbers. The two pairs $(a,b)$ and $(a',b')$ are double numbers, or numbers belonging to the algebra $\mathbb R \oplus \mathbb R$. The number $\delta$ on the other hand is quite exotic. First of all, $\delta$ satisfies $\delta^2 = 0$. Additionally, a number of the form $(a',b')\delta$ is essentially in its simplest form, but a number of the form $\delta(a',b')$ simplifies to $(b',a')\delta$. $\delta$ therefore acts as a swapping operator. The swapping operation here is essentially identical to the operation we defined as $*_{-1}$ over $\mathbb R \oplus \mathbb R$.

The algebra $\cl_{1,0,1}(\mathbb R)$ is still not a $*$-algebra because we have not equipped it with an involution. We will define an involution $*_1$ by $((a,b) + (a',b')\delta)^{*_1} = (a,b) + (b',a')\delta$. We will denote this more simply as $*$ unless this results in ambiguity. We will denote the corresponding $*$-algebra as \newcommand{\svdalg}{(\cl_{1,0,1}(\mathbb R),*_1)}$\svdalg$.

How should we write a matrix over $\svdalg$? We can write it as $(M,K)+(M',K')\delta$ with the obvious meaning. What is the adjoint operation over $\svdalg$? It is $((M,K)+(M',K')\delta)^* = (M^T,K^T) + ((K')^T,(M')^T)\delta$. What then is a $\svdalg$-unitary matrix? It is of the form $(U,V)(I + (K,-K^T)\delta)$ where $U$ and $V$ are orthogonal matrices. What is then a self-adjoint matrix? It is of the form $(H,K) + (M,M^T)\delta$ where $H=H^T$ and $K=K^T$. We can consider an \emph{infinitesimal} self-adjoint matrix to be of the form $(M,M^T)\delta$ because $\delta$ behaves like an infinitesimal.

We now consider the special case of the ``spectral theorem'' for only \emph{infinitesimal} self-adjoint matrices. This would be a canonical form for an infinitesimal self-adjoint matrix under unitary similarity. Consider an infinitesimal self-adjoint matrix $(M,M^T)\delta$ conjugated by a unitary matrix $(U,V)(I + (K,-K^T)\delta)$; this is written as $(U,V)(I + (K,-K^T)\delta) \times (M,M^T)\delta \times ((U,V)(I + (K,-K^T)\delta))^*$, and simplifies to $(U M V^T,V M^T U^T)\delta$. The canonical form must therefore be the same as the $\mathbb R$-singular value decomposition. A unification is thus achieved between a special case of the spectral theorem over $\svdalg$ (over only the infinitesimal self-adjoint matrices) and the singular value decomposition over $\mathbb R$. Note that the reasoning is valid with respect to any $*$-field in place of $\mathbb R$ (including $\mathbb C$).

Appendix remark: Note that we could have defined another involution $*_{-1}$ by $((a,b) + (a',b')\delta)^{*_{-1}} = (a,b) - (b',a')\delta$ and considered the $*$-algebra $(\cl_{1,0,1}(\mathbb R),*_{-1})$ instead. But this is in fact redundant because $(\cl_{1,0,1}(\mathbb R),*_{-1})$ is isomorphic to $\svdalg$. The isomorphism is given by $(a,b) + (a',b')\delta \mapsto (a,b) + (a',-b')\delta$, and works in either direction.

\begin{proposition}
    \label{svd-algl-iso}
    $\svdalg$ is isomorphic to $(\cl_{1,0,1}(\mathbb R),*_{-1})$.
\end{proposition}

\subsection{The ``Takagi $*$-algebra''}

Consider the Clifford algebra $\cl_{0,1,1}(\mathbb R)$, which we will equip with a certain involution. To make the reader's life easier, we will describe this algebra explicitly. The elements of this algebra are all of the form
$$(a + b i) + (a' + b'i)\delta$$
where $a,b,a',b'$ are all real numbers. The two components $a + b i$ and $a' + b'i$ are complex numbers. The number $\delta$ on the other hand is quite exotic. First of all, $\delta$ satisfies $\delta^2 = 0$. Additionally, a number of the form $(a' + b' i)\delta$ is essentially in its simplest form, but a number of the form $\delta(a' + b' i)$ simplifies to $(a' - b' i)\delta$. $\delta$ therefore acts as a complex conjugation operator.

The algebra $\cl_{0,1,1}(\mathbb R)$ is still not a $*$-algebra because we have not equipped it with an involution. We will define an involution $*_1$ by $((a + b i) + (a' + b' i)\delta)^{*_1} = (a - b i) + (a' + b' i)\delta$. We will denote this more simply as $*$ unless this results in ambiguity. We will denote the corresponding $*$-algebra as \newcommand{\takalg}{(\cl_{0,1,1}(\mathbb R),*_1)}$\svdalg$.

How should we write a matrix over $\takalg$? We can write it as $M+K\delta$ with the obvious meaning. What is the adjoint operation over $\takalg$? It is $(M + K\delta)^* = M^* + K^T\delta$ where the $M^*$ denotes conjugate-transpose of a complex matrix. What then is a $\takalg$-unitary matrix? It is of the form $U(I + K\delta)$ where $U$ is complex unitary and $K$ is skew complex-symmetric. What is then a self-adjoint matrix over $\takalg$? It is of the form $H + S\delta$ where $H=H^*$ and $S = S^T$. We can consider an \emph{infinitesimal} self-adjoint matrix to be of the form $S\delta$ where $S$ is complex-symmetric because $\delta$ behaves like an infinitesimal.

We now consider the special case of the ``spectral theorem'' for only \emph{infinitesimal} self-adjoint $\takalg$-matrices. This would be a canonical form for an infinitesimal self-adjoint $\takalg$-matrix under unitary similarity. Consider an infinitesimal self-adjoint $\takalg$-matrix $S\delta$ (where $S = S^T$) conjugated by a $\takalg$-unitary matrix $U(I + K\delta)$; this is written as $U(I + K\delta) \times S\delta \times (U(I + K\delta))^*$, and simplifies to $U S U^T \delta$. The canonical form must therefore be the same as the Takagi decomposition. A unification is thus achieved between a special case of the spectral theorem over $\takalg$ (over only the infinitesimal self-adjoint matrices) and the Takagi decomposition for symmetric matrices over $\mathbb C$.

\subsection{The ``skew-Takagi $*$-algebra''}

Consider the same algebra $\cl_{0,1,1}(\mathbb R)$ as in the previous subsection, but with a different involution. Define $*_{-1}:\cl_{0,1,1}(\mathbb R) \to \cl_{0,1,1}(\mathbb R)$ by $(w + z\delta)^{*{-1}} = w^* - z^*\delta$, where in the previous definition we instead had $w^* + z^*\delta$. An infinitesimal self-adjoint matrix is now of the form $S \delta$ where $S$ is \emph{skew complex-symmetric} instead of \emph{complex-symmetric} as before (i.e. we need $S = -S^T$). Applying a unitary similarity to it simplifes to $U S U^T\delta$ where $U$ is some complex unitary matrix. The spectral theorem for infinitesimal self-adjoint matrices over \newcommand{\skewtakalg}{(\cl_{0,1,1}(\mathbb R), *_{-1})}$\skewtakalg$ is thus the same as the skew-Takagi decomposition.

\subsection{Remark about dualities between matrix decompositions}

The above is surprising because the Takagi decomposition is usually seen as a special case of the SVD. What we've uncovered though is instead a duality between them. Takagi is not a special case of SVD, but is its dual. This duality suggests that somehow, whatever ``works'' (let's say in numerical computing) for the SVD should also work over Takagi. Since the SVD is more thoroughly studied, this suggests that one can transfer the numerical theory of the SVD wholesale onto the Takagi and skew-Takagi decompositions.

The skew-Takagi decomposition is obviously dual to the Takagi decomposition in a different way to how it is dual to the SVD.

Notice that when diagonalising a self-adjoint matrix $H$ over $\mathbb C$, the following steps are taken:
\begin{enumerate}
    \item Find an eigenvector $v$ of $H$.
    \item Find a basis $B$ for the orthogonal complement of $v$, written $v^\perp$.
    \item Restrict $H$ to $v^\perp$ by using the basis $B$.
    \item Return to the first step.
\end{enumerate}
The trick is to realise that the SVD and Takagi follow the same plan. In the case of the SVD, given an arbitrary $\mathbb R$-matrix $M$, the matrix $(M,M^T)\delta$ is a self-adjoint matrix in $\svdalg$. A left and right singular vector of $M$, when paired together as $(u,v)$, is actually an eigenvector of $(M,M^T) \delta$. Conversely, an eigenvector of $(M,M^T)\delta$ is immediately reducible to the $\svdalg$-vector $(u,v)$ where $u$ and $v$ are singular vectors of $M$. Steps 2 and 3 are simple. In step 2, the orthogonal complement of $(u,v)$ is the Cartesian product of the orthogonal complements of $u$ and $v$. In symbols: $(u,v)^\perp = u^\perp \times v^\perp$. Thus, given an ordered basis $(b_1, b_2, \dotsc, b_n)$ of $u^\perp$ and an ordered basis $(c_1, c_2, \dotsc, c_n)$ of $v^\perp$, we can obtain an ordered basis $((b_1,c_1), (b_2,c_2), \dotsc, (b_n,c_n))$ of $(u,v)^\perp$.

\newcommand{\unwind}{\operatorname{unwind}}

Step 1 is usually the most complicated, but there are approaches which sometimes work, depending on the $*$-algebra. A good approach can be called the \emph{unpack-and-unwind} method (for lack of a better name). Given a $*$-algebra $(\mathcal A, *)$, and a sub-$*$-algebra $\mathcal B$, we define a function ``$\unpack$'' that sends a $(\mathcal A, *)$-matrix to a $(\mathcal B,*)$-matrix. For example, consider the matrix $(M,M^T)\delta$ over $\svdalg$, and let $\mathcal B$ be the dual numbers. We have that $\unpack((M,M^T)\delta) = \begin{pmatrix} 0 & M \varepsilon\\ M^T\varepsilon & 0\end{pmatrix}$. Ignoring the $\varepsilon$, we can easily obtain an eigenvector $v$ of $\unpack((M,M^T)\delta)$ using the $\mathbb R$-spectral theorem. Notice though that an eigenvector of $\unpack((M,M^T)\delta)$ is not an eigenvector of $(M,M^T)\delta$. We define another function called $\unwind$. $\unwind$ sends an $\mathcal A$-vector to a $\mathcal B$-vector in such a way that $\unwind(Kv) = \unpack(K) \unwind(v)$. For our eigenvector $v$ of $\unpack((M,M^T)\delta)$, we can verify that $v = \unwind(v')$ for some $\svdalg$-vector $v'$. We then have that $\unwind(\lambda v') = \lambda \unwind(v') = \unpack((M,M^T)\delta) \unwind(v') = \unwind((M,M^T)\delta v')$. By the injectivity of $\unwind$, we cancel to get that $v'$ is an eigenvector of $(M,M^T)\delta$.

Note that all of this is \emph{not} the same thing as reducing SVD or Takagi to the diagonalisation of some suitable self-adjoint $\mathbb C$-matrix. These tricks are well-known, suboptimal, and they require additional post-processing steps which the method we're proposing doesn't.

\section{Infinitesimal spectral theorems, and the resulting equivalence relations which give rise to numerous matrix decompositions}

\subsection{Understanding $\delta$ in general}

The algebras above -- that is, $\cl_{0,1,1}(\mathbb R)$ and $\cl_{1,0,1}(\mathbb R)$, and without considering involutions -- could be obtained from $\mathbb C (\cong \cl_{0,1,0}(\mathbb R))$ or $\mathbb R \oplus \mathbb R (\cong \cl_{1,0,0}(\mathbb R))$ respectively in multiple ways. On the one hand, they could be obtained by introducing a nilsquare generator $\delta$ into the Clifford algebra to obtain a larger Clifford algebra.

A broader direction of generalisation is also apparent: We introduced an element $\delta$ into some algebra $A$ over some field $\mathbb F$ such that $\delta z = \phi(z) \delta$ where $\phi:A \to A$ is some function. This is a lot like the Cayley-Dickson construction. What properties should $\phi$ satisfy?

To have associativity, we will need to have $\phi(wz) = \phi(w) \phi(z)$. Why? Because by associativity, we have that $\phi(wz)\delta = \delta(wz) = (\delta w) z = (\phi(w) \delta) z = \phi (w) (\delta z) = \phi(w) (\phi(z) \delta) = (\phi(w) \phi(z)) \delta$. This already rules out the standard quaternion involution, or the matrix transpose, as possible instantiations of $\phi$.

$\phi$ will need to be linear over the underlying field. Why? Because $\phi(w + \lambda z)\delta = \delta(w + \lambda z) = \delta w + \delta \lambda z = \phi(w)\delta + \lambda \phi(z) \delta = (\phi(w) + \lambda \phi(z))\delta$ where $\lambda$ is in the field.

Notice though that it would not suffice for $\phi$ to be an antiautomorphism. The usual quaternion involution is actually an antiautomorphism and not an automorphism -- as is the matrix transpose -- and so could not be used as a $\phi$.

We can now state:

\begin{proposition}
    Let $A$ be an algebra over a field $\mathbb F$. Let $A[[\delta]]$ denote the algebra obtained by adjoining to $A$ an element $\delta$ such that $\delta^2 = 0$ and $\delta z = \phi(z) \delta$ for some function $\phi: A \to A$ and all $z \in A$. The result is an associative algebra over $\mathbb F$ if and only if $\phi$ is an algebra automorphism of $A$.
\end{proposition}

\subsection{The unifying list}

Each matrix decomposition we consider presents a canonical form for a matrix under some equivalence relation. In our statement of conjecture \ref{conjecture}, we speculate that this canonical form has a particularly nice structure for the general type of equivalence relation we consider here.

Some matrix decompositions are precisely equivalent to spectral theorems over $*$-algebras. We've seen this with the spectral theorem for $\doublenontriv$, which is precisely equivalent to the Jordan decomposition. But oftentimes, the equivalence only holds for so-called \emph{infinitesimal matrices}, with the non-infinitesimal spectrum theorem being strictly more general than is needed to describe a classic matrix decomposition.

Given a $*$-algebra $(R,*)$, an endomorphism $\phi: R \to R$, and a value $s \in \{-1,+1\}$, we can define the $*$-algebra $R[[\delta]]$ whose elements are of the form $w + z \delta$ ($w, z \in R$), where $\delta z = \phi(z) \delta$, and where $(w + z \delta)^* = w^* + s \phi(z^*) \delta$.

A matrix $M$ over the $*$-algebra $R[[\delta]]$ is called \emph{infinitesimal} if it is equal to $K \delta$ for some $(R,*)$-matrix $K$. An \emph{infinitesimal self-adjoint matrix} $H$ is one which is infinitesimal and self-adjoint.

The unitary-similarity relation specialised to infinitesimal self-adjoint matrices is equivalent (depending on $R, *, \phi, s$) to numerous equivalence relations on $\mathbb R$-matrices, with a likely spectral theorem for each of them:

\begin{itemize}
    \item Ordinary matrix similarity, whose canonical form is the Jordan decomposition, if $R = \mathbb R \oplus \mathbb R$, $(a,b)^* = (b,a)$, $s = 1$ and $\phi(a,b) = (a,b)$.
    \item The equivalence relation $M \sim U M V^T$ (for $M$ real, and $U$ and $V$ real-unitary), whose canonical form is the SVD, if $R = \mathbb R \oplus \mathbb R$, $(a,b)^* = (a,b)$, $s = 1$ and $\phi(a,b) = (b,a)$.
    \item The equivalence relation $M \sim V M V^T$ for $M$ complex-symmetric and $V$ complex-unitary, whose canonical form is the Takagi decomposition, if $R = \mathbb C$, $(a + bi)^* = a - bi$, $s = 1$ and $\phi(a + bi) = a - bi$.
    \item The equivalence relation $M \sim V M V^T$ for $M$ complex skew-symmetric and $V$ complex-unitary, whose canonical form is the skew-Takagi decomposition if $R = \mathbb C$, $(a + bi)^* = a - bi$, $s = -1$ and $\phi(a + bi) = a - bi$.
    \item Consider four symmetric bilinear forms $B_1, B'_1, B_2, B'_2 : V \otimes V \to \mathbb R$, the first two being non-degenerate, over an $\mathbb R$-vector space $V$. We say that $(B_1, B_2)$ is equivalent to $(B'_1,B'_2)$ if there is an invertible linear operator $P : V \to V$ such that for all $i \in \{1,2\}$ we have $B_i(P(u),P(v)) = B'_i(u,v)$. This equivalence relation occurs when $R = \mathbb R \oplus \mathbb R$, $(a,b)^* = (b,a)$, $s = 1$ and $\phi(a,b) = (b,a)$. We don't consider this one further.
    \item The equivalence relation $M \sim V M V^T$ for $M$ dual-number symmetric and $V$ dual-number unitary, whose canonical form is a dual-Takagi decomposition, if $R = \mathbb D$, $(a + b\varepsilon)^* = a + b\varepsilon$, $s = 1$ and $\phi(a + b\varepsilon) = a - b\varepsilon$.
\end{itemize}

There are other special cases of spectral theorems over $*$-algebras that are equivalent to other matrix decompositions.

\section{The unpack-and-unwind method for computing and verifying existence of some classic matrix decompositions}

The following theorems aren't new, but are proved using the same technique. These are special cases of spectral theorems which we will prove fully later. Note that by $\mathbb C$, we will mean the $*$-algebra of complex numbers equipped with their standard involution: $(a + bi)^* = a - bi$. We won't write this $(\mathbb C, *_{-1})$ for the sake of readability, but be aware that the complex numbers may (outside of this paper) be equipped with a different involution.

Below, we piecemeal define operations we call $\unpack$ and $\unwind$. The $\unwind$ operation acts on column vectors, and is at least \emph{partially} a mere change of scalars. The $\unpack$ operation is then defined such that $\unpack(M)\unwind(v) = \unwind(M v)$ for all $v$. We say that $\unwind$ is only partially a change of scalars, because it changes the scalar $*$-algebra to one of its subalgebras, but in such a way that for instance the ``length squared'' of a vector might be preserved.

For example, the following would be a \emph{bad way} to define $\unwind$ from $\mathbb C$-vectors to $\mathbb R$-vectors: $\unwind(u + v i) = (u + v, u)^T$. The problem with this definition is that if $z = u + v i$, then $z^* z \neq \unwind(z)^* \unwind(z)$ according to this proposal.

In some instances below, we unfortunately cannot define $\unwind$ such that $v^* v = \unwind(v)^* \unwind(v)$, but a seemingly natural definition of $\unwind$ is still possible, and we still give one. Such difficulties arise in any ring which contains nontrivial idempotents (i.e. solutions to $x^2 = x$ which are not either $0$ or $1$): For example, in the ring $\mathbb R \oplus \mathbb R$.

In general, the $\unwind$ and $\unpack$ operations ought to satisfy:
$$\begin{aligned}
  \unwind(u + v) &= \unwind(u) + \unwind(v) \\
  \unwind(A v) &= \unpack(A) \unwind(v) \\
  \unpack(A B) &= \unpack(A) \unpack(B) \\
  \unpack(A^*) &= \unpack(A)^*\\
  \text{$\unwind$ is bijective}
\end{aligned}$$

\begin{proposition}[Takagi decomposition] \label{takagi}
    Given a complex-symmetric matrix $M$ (i.e. one which satisfies $M = M^T$, and not to be confused with a self-adjoint $\mathbb C$-matrix) there exists a $\mathbb C$-unitary matrix $U$ and an $\mathbb R$-diagonal matrix $D$ such that $M = UDU^T$. (Note that $U^T \neq U^{-1}$).
\end{proposition}
\begin{proof}
    Observe that while $M$ is not Hermitian, we may introduce a new scalar $\delta$ such that $\delta^2 = 0, \delta i = -i \delta$ and $\delta^* = \delta$. We see that $M\delta$ is now a $\takalg$-matrix, and is indeed self-adjoint (over $\takalg$) because $(M\delta)^* = \delta^* M^* = \delta \overline M^T = M^T \delta = M \delta$. We might then hope to unitarily diagonalise $M\delta$. We seek to show that a $\takalg$-unitary diagonalisation of $M \delta$ -- should it exist -- gives rise to a Takagi decomposition of $M$. To see this, we begin with assuming that a $\takalg$-unitary diagonalisation exists: $M \delta = U D U^*$ for $\takalg$-matrices $U$ and $D$ where $U$ is $\takalg$-unitary and $D$ is $\takalg$-diagonal. We write each of $U$ and $D$ as:
    $$\begin{aligned}
        U = U_0 + U' \delta \\
        D = D_0 + D' \delta
    \end{aligned}$$
    where the four matrices $U_0, U', D_0, D'$ are all complex matrices, and $U_0$ is $\mathbb C$-unitary. Substituting these into $M \delta = U D U^*$ and simplifying gives $M \delta = U_0 D' U_0^T \delta$. Taking the $\delta$ component finally gives $M = U_0 D' U_0^T$. This is a Takagi decomposition of $M$. Therefore, our problem is reduced to finding the unitary diagonalisation of $M \delta$.
    
    We define $\unpack(A + Bi + C\delta + D i \delta) := \begin{pmatrix}A - C \varepsilon & -B + D\varepsilon \\ B + D \varepsilon & A + C \varepsilon \end{pmatrix}$ (where $A, B, C, D$ are arbitrary $\mathbb R$-matrices of equal dimensions) and $\unwind(a + bi + c\delta + d i \delta) := \begin{pmatrix}a + c \epsilon \\ b + d \varepsilon \end{pmatrix}$ (where $a, b, c, d$ are $\mathbb R$-vectors of equal dimensions). We observe that $\unpack(MK) = \unpack(M) \unpack(K)$ and $\unwind(Mv) = \unpack(M) \unwind(v)$.
    
    We have that $\unpack(M \delta) = \begin{pmatrix} -\Re(M) & \Im(M) \\ \Im(M) & \Re(M)\end{pmatrix} \varepsilon$. Ignoring the $\varepsilon$, we have a self-adjoint $\mathbb R$-matrix. By the $\mathbb R$-spectral theorem, we obtain an $\mathbb R$-eigenvector $v$ of $\unpack(M \delta)$. Unfortunately, $v$ is not an eigenvector of $M \delta$, but only of $\unpack(M \delta)$. By the expression for $\unwind$ above, we have that there is a $\mathbb C$-vector $v'$ such that $\unwind(v') = v$. We see that $v'$ is an eigenvector of $M \delta$.

    Since $M \delta$ is self-adjoint, it maps the orthogonal complement of $v'$ to itself. We can define a matrix representation of $M \delta$ over this subspace. Since $v'$ is a $\mathbb C$-vector, a $\takalg$-orthonormal basis $B$ for $(v')^\perp$ is obtained simply from the $\mathbb C$-orthogonal complement of $v'$. The matrix representation of $M \delta$ over $(v')^\perp$ is obtained using the basis $B$. The result is a self-adjoint $\takalg$-matrix $M' \delta$ with one less dimension than $M \delta$. We repeat by finding another eigenvector, restricting to the orthogonal complement of \emph{that}, etc. Formally, this proof is by induction.
\end{proof}

\begin{proposition}[Singular Value Decomposition] \label{svd}
    Given a square $\mathbb R$-matrix $M$ there exists a pair of $\mathbb R$-unitary matrices $U$ and $V$, and an $\mathbb R$-diagonal matrix $D$, such that $UDV^T = M$.
\end{proposition}
\begin{remark}
    Before we begin the proof, we provide a justification for providing a new proof of the singular value decomposition, other than that we can use the same method for other decompositions. The approach usually taught to students for proving or computing the SVD, which revolves around diagonalising $M^T M$ does \emph{not work over the ring of dual numbers}. This is because the eigenvalues of $M^T M$ are the squares of the singular values of $M$. The singular values of a dual number matrix may square to zero without actually being zero. This provides a simplified model of numerical underflow in inexact floating point or fixed point computation. While our method below uses a well-known block matrix $\begin{pmatrix}0 & M \\ M^T & 0 \end{pmatrix}$ whose singular values are plus-and-minus the singular values of $M$, the argument below succeeds in building an existence-of-SVD proof around it, where we ensure that the unitary diagonalisation we obtain has the block structure we need. For the dual-number spectral theorem(s) and initial proofs of the corresponding SVD(s), see \cite{gutin_generalizations_2021}. Note that we say ``(s)'' because there is a separate dual-number spectral theorem (and SVD) for each of the two possible involutions over the dual numbers.
\end{remark}
\begin{proof}
    Observe that $(M, M^T) \delta$ is self-adjoint over $\svdalg$. We might then reasonably hope to obtain a unitary diagonalisation of $(M,M^T)\delta$. We seek to show that a $\svdalg$-unitary diagonalisation of $(M, M^T) \delta$ -- should it exist -- gives rise to a singular value decomposition of $M$. To see this, we begin with assuming that a $\svdalg$-unitary diagonalisation exists: $(M,M^T) \delta = W D W^*$ for $\svdalg$-matrices $W$ and $D$ where $W$ is $\svdalg$-unitary and $D$ is $\svdalg$-diagonal. We write each of $W$ and $D$ as:
    $$\begin{aligned}
        W = (U,V) + (U',V') \delta \\
        D = (E,F) + (\Sigma, \Sigma') \delta
    \end{aligned}$$
    where the eight matrices $U, V, U', V', E, F, \Sigma, \Sigma'$ are all real matrices, and the two matrices $U$ and $V$ are $\mathbb R$-unitary. Substituting these into $(M,M^T) \delta = W D W^*$ and simplifying gives $\Sigma' = \Sigma^T$ and $(M,M^T) \delta = (U \Sigma V^T, V \Sigma U^T) \delta$. Taking the $\delta$ component finally gives $(M,M^T) = (U \Sigma V^T, V \Sigma U^T)$, which implies $M = U \Sigma V^T$. This is an SVD of $M$. Therefore, our problem is reduced to finding the unitary diagonalisation of $(M,M^T) \delta$.
    
    We define $\unpack((A,B) + (C,D)\delta) := \begin{pmatrix}A & C \varepsilon \\ D \varepsilon & B \end{pmatrix}$ (where $A, B, C, D$ are arbitrary $\mathbb R$-matrices of equal dimensions) and $\unwind((a,b) + (c,d) \delta) := \begin{pmatrix}a + d \varepsilon \\ b + c \varepsilon \end{pmatrix}$ (where $a, b, c, d$ are $\mathbb R$-vectors of equal dimensions). We observe that $\unpack(MK) = \unpack(M) \unpack(K)$ and $\unwind(Mv) = \unpack(M) \unwind(v)$.
    
    We have that $\unpack((M,M^T) \delta) = \begin{pmatrix} 0 & M \\ M^T & 0\end{pmatrix} \varepsilon$. Ignoring the $\varepsilon$, we have a self-adjoint $\mathbb R$-matrix. By the $\mathbb R$-spectral theorem, we obtain an $\mathbb R$-eigenvector $w$ of $\unpack((M,M^T) \delta)$. Unfortunately, $w$ is not an eigenvector of $(M,M^T) \delta$, but only of $\unpack((M,M^T) \delta)$. By the expression for $\unwind$ above, we have that there is a $\mathbb R$-vector $(u,v)$ such that $\unwind((u,v)) = w$. We see that $(u,v)$ is an eigenvector of $(M,M^T) \delta$.

    Since $(M,M^T) \delta$ is self-adjoint, it maps the orthogonal complement of $(u,v)$ to itself. We can define a matrix representation of $(M,M^T) \delta$ over this subspace. Since $(u,v)$ is a $\mathbb R \oplus \mathbb R$-vector, a $\svdalg$-orthonormal basis $B$ for $(u,v)^\perp$ is obtained simply as the $\mathbb R$-orthogonal complement of $v'$. The matrix representation of $(M,M^T) \delta$ over $(v')^\perp$ is obtained using the basis $B$. The result is a self-adjoint $\svdalg$-matrix $(M',(M')^T) \delta$ with one less dimension than $(M,M^T) \delta$. We repeat by finding another eigenvector, restricting to the orthogonal complement of \emph{that}, etc. Formally, this proof is by induction.
\end{proof}

\begin{proposition}[skew-Takagi decomposition] \label{skew-takagi}
    Given a complex skew-symmetric matrix $M$ (i.e. one which satisfies $M = -M^T$, and not to be confused with a skew-Hermitian $\mathbb C$-matrix) there exists a $\mathbb C$-unitary matrix $U$ and an $\mathbb R$-matrix $D = D_1 \oplus D_2 \oplus \dotsb \oplus D_k$ such that $D_i$ is either of the form $\begin{pmatrix} 0 & -\mu \\ \mu & 0 \end{pmatrix}$ or $(0)$, and $M = UDU^T$. (Note that $U^T \neq U^{-1}$).
\end{proposition}
\begin{proof}
    Observe that while $M$ is not Hermitian, we may introduce a new scalar $\delta$ such that $\delta^2 = 0, \delta i = -i \delta$ and $\delta^* = -\delta$. We see that $M\delta$ is now a $\skewtakalg$-matrix, and is indeed self-adjoint (over $\skewtakalg$) because $(M\delta)^* = \delta^* M^* = -\delta \overline M^T = -M^T \delta = M \delta$. We can now hope to employ some analogue of the spectral theorem. It's easily seen that if for a matrix $K$ we have that $K \delta$ is unitarily similar to $M \delta$, then there is a $U$ such that $UMU^T = K$. We intend to find a canonical $K$.
    
    We define $\unpack(A + Bi + C\delta + D i \delta) := \begin{pmatrix}A + C \varepsilon & -B - D\varepsilon \\ B - D \varepsilon & A + C \varepsilon \end{pmatrix}$ (where $A, B, C, D$ are arbitrary $\mathbb R$-matrices of equal dimensions) and $\unwind(a + bi + c\delta + d i \delta) := \begin{pmatrix}a + c \epsilon \\ b + d \varepsilon \end{pmatrix}$ (where $a, b, c, d$ are $\mathbb R$-vectors of equal dimensions). We observe that $\unpack(MK) = \unpack(M) \unpack(K)$ and $\unwind(Mv) = \unpack(M) \unwind(v)$.
    
    We have that $\unpack(M \delta) = \begin{pmatrix} \Re(M) & -\Im(M) \\ -\Im(M) & \Re(M)\end{pmatrix} \varepsilon$. Ignoring the $\varepsilon$, we have a skew-symmetric $\mathbb R$-matrix. By the skew-symmetric $\mathbb R$-spectral theorem, we obtain a pair of $\mathbb R$-eigenvectors $u$ and $v$ such that there exists a $\mu \in \mathbb R$ for which $\unpack(M) u = -\mu v$ and $\unpack(M)v = \mu u$. Unfortunately, $u$ and $v$ are not vectors over the same algebra as $M \delta$. By the expression for $\unwind$ above, we have that there are $\mathbb C$-vectors $u'$ and $v'$ such that $\unwind(u') = u$ and $\unwind(v') = v$. We see that $M \delta v' = - \mu v'$ and $M \delta u' = \mu v'$.

    \newcommand{\spn}{\operatorname{span}}
    Since $M \delta$ is self-adjoint, it maps the orthogonal complement of $\spn \{u', v'\}$ to itself. We can define a matrix representation of $M \delta$ over this subspace. Since $u'$ and $v'$ are $\mathbb C$-vectors, a $\skewtakalg$-orthonormal basis $B$ for $\spn \{u', v'\}^\perp$ is obtained simply from the $\mathbb C$-orthogonal complement of $\spn \{u', v'\}$. The matrix representation of $M \delta$ over $\spn \{u', v'\}^\perp$ is obtained using the basis $B$. The result is a self-adjoint $\skewtakalg$-matrix $M' \delta$ with one less dimension than $M \delta$. We repeat the above with $M'\delta$. Formally, this proof is by induction.
\end{proof}

\begin{proposition}[quaternion skew-spectral theorem] \label{skew-quaternion-spectral}
    Given a skew-Hermitian quaternion matrix $M$ (i.e. one which satisfies $M = -M^*$, where we use the standard quaternion involution) there exists a $\mathbb H$-unitary matrix $U$ and an $\mathbb R$-matrix $D = D_1 \oplus D_2 \oplus \dotsb \oplus D_k$ such that $D_i$ is either of the form $\begin{pmatrix} 0 & -\mu \\ \mu & 0 \end{pmatrix}$ or $(0)$, and $M = UDU^*$.
\end{proposition}
\begin{proof}
    We are not able to employ the same trick as for a skew-Hermitian matrix over the complex numbers (with their usual involution). We employ unpack-and-unwind instead.
    
    We define $\unpack(A + Bi + Cj + D k) := \begin{pmatrix}A + B i & -C + D i \\ C + D i & A - B i \end{pmatrix}$ (where $A, B, C, D$ are arbitrary $\mathbb R$-matrices of equal dimensions) and $\unwind(a + b i + c j + d k) := \begin{pmatrix}a + bi \\ c + d i \end{pmatrix}$ (where $a, b, c, d$ are $\mathbb R$-vectors of equal dimensions). We observe that $\unpack(MK) = \unpack(M) \unpack(K)$ and $\unwind(Mv) = \unpack(M) \unwind(v)$.
    
    We have that $\unpack(M)$ is a skew-symmetric $\mathbb R$-matrix. By the skew-symmetric $\mathbb R$-spectral theorem, we obtain a pair of $\mathbb R$-eigenvectors $u$ and $v$ such that there exists a $\mu \in \mathbb R$ for which $\unpack(M) u = -\mu v$ and $\unpack(M) v = \mu u$. Unfortunately, $u$ and $v$ are not vectors over the same algebra as $M$. By the expression for $\unwind$ above, we have that there are $\mathbb C$-vectors $u'$ and $v'$ such that $\unwind(u') = u$ and $\unwind(v') = v$. We see that $M \delta v' = - \mu v'$ and $M \delta u' = \mu v'$.
    \newcommand{\spn}{\operatorname{span}}
    Since $M$ is a skew-Hermitian, it maps the orthogonal complement of $\spn \{u', v'\}$ to itself. We can define a matrix representation of $M$ over this subspace. We know that all modules over the quaternions admit a basis, and this basis can always be orthonormalised. We restrict $M$ to the orthogonal complement of $\spn \{u', v'\}$ by use of a basis, and obtain a matrix $M'$. Formally, this proof is by induction.
\end{proof}

There are many more examples of this trick, for instance involving some matrix decompositions over the dual numbers, reducing them to the dual-number spectral theorem \cite{gutin_generalizations_2021}.

\section{Proving the spectral theorem for the ``SVD $*$-algebra'' $\svdalg$}

In this section, we prove a spectral theorem for $\svdalg$.

First though, we must define the $\unpack$ and $\unwind$ operations:
$$\begin{aligned}
  \unpack\left((A,B) + (B',A')\delta\right) &= \begin{bmatrix}A & A'\varepsilon \\ B'\varepsilon & B\end{bmatrix}, \\
  \unwind\left((u,v) + (v',u')\delta\right) &= \begin{bmatrix}u + \varepsilon u' \\ v + v'\varepsilon\end{bmatrix}.
\end{aligned}$$
We see that
$$\begin{aligned}
  \unpack(X Y) &= \unpack(X) \unpack(Y),\\
  \unpack(X + Y) &= \unpack(X) + \unpack(Y),\\
  \unwind(X v) &= \unpack(X) \unwind(v),\\
  \unwind(u + v) &= \unwind(u) + \unwind(v),\\
  \unpack(X^*) &= \unpack(X)^*.
\end{aligned}$$
\begin{paragraph}{Some preliminaries on notation:} In the following arguments, we sometimes conflate a dual number $a + b \varepsilon$ with a member of $\svdalg$ of the form $a + b \delta$. We define $\operatorname{st}(a + b\varepsilon) = a$ (the ``standard part'') and $\operatorname{nst}(a + b \varepsilon) = b$ (the ``non-standard part''). Be aware that when we write $(a,b)$, we mean a member of $\svdalg$, and not a row vector -- we don't presently expect this to cause confusion. When we write $(M,K)$ (or some other capital letters), we mean a matrix over $\svdalg$ of the form $M(1,0) + K(0,1)$.
\end{paragraph}

\begin{paragraph}{Some facts we're assuming:} We make extensive use of the fact that every dual number matrix satisfying $S = S^T$ also satisfies a form of the spectral theorem: We have that $S = U D U^T$ is true for a \emph{unique} diagonal matrix $D$ over the dual numbers, and some orthogonal matrix $U$ over the dual numbers (i.e. satisfying $U^T = U^{-1}$).
\end{paragraph}

\begin{lemma} \label{semi-eigenvalue-means-eigenvalue}
  If a self-adjoint $\svdalg$-matrix $H$ admits a pair of vectors $w_L$ and $w_R$ such that:
  \begin{enumerate}
    \item $w_L = w_L (1,0)$ and $w_R = w_R (0,1)$,
    \item $w_L^* w_L = (1,0)$ and $w_R^* w_R = (0,1)$,
    \item There exist real numbers $\lambda_L$ and $\lambda_R$ for which $H w_L = w_L \lambda_L$ and $H w_R = w_R \lambda_R$,
  \end{enumerate}
  then $w = w_L + w_R$ is a unit eigenvector of $H$ with eigenvalue $(\lambda_L, \lambda_R)$.
\end{lemma}
\begin{proof}
  From 1, we may expand $w_L$ to $w_L = (u, 0) + \delta(u', 0)$ and $w_R$ to $w_R = (0, v) + \delta(0, v')$. Then $w_L + w_R = (u,v) + \delta(u',v')$. Let $w = w_L + w_R$. It's easy to see that $w$ is an eigenvector of $H$ of eigenvalue $(\lambda_L, \lambda_R)$. It remains to consider $w^* w$.
  
  We have that $w^* w = A + B\delta$ for some $A$ and $B$. We know that $A = 1$ because $u^* u = v^* v = 1$ (from item 2). We also get that $B$ is real because $B\delta = w^* w - 1 = (w^*w - 1)^* = (B\delta)^*$.

  In fact, $B = 0$, as we will now show. Let $\mu$ be the eigenvalue of $H$ corresponding to $w$. We have $(1+B\delta)\mu = w^* w \mu = w^* H w = (H w)^* w = (w \mu)^* w = \mu^* w^* w = \mu (1 + B\delta)$. Since $\mu = (\lambda_L, \lambda_R)$ and $(1+B\delta)\mu = \mu (1 + B\delta)$, we get $B(\lambda_L - \lambda_R) = 0$. Since furthermore $\lambda_L \neq \lambda_R$, we have that $B = 0$.
\end{proof}

\begin{lemma} \label{dual-eigenvalue-means-eigenvalue}
  Let $\unwind(v)$ be an eigenvector of $\unpack(H)$ (for a self-adjoint $\svdalg$ matrix $H$) with eigenvalue $\lambda$ that is not real. We have that the normalised vector $w = v (v^* v)^{-1/2}$ satisfies $w^* w = 1$ and $H w = w \lambda$.
\end{lemma}
\begin{proof}
  Observe that $H v = v \lambda$. From this, we get that $H (v(1,-1)) = (v(1,-1)) \overline\lambda$. Expand $v$ to $v = (a,b) + \delta(c,d)$. Expand $\lambda$ to $\lambda_0 + \lambda'\varepsilon$.
  
  \begin{paragraph}{We seek to show that $a^* a = b^* b$:} We have that $v^* v(1,-1) \overline\lambda = v^* H v(1,-1) = (H v)^* v(1,-1) = (v \lambda)^* v(1,-1) = \lambda v^* v(1,-1)$. We therefore have that $v^* v(1,-1) \overline\lambda = \lambda v^* v(1,-1)$. Substituting our expansion of $\lambda$ into this gives $v^* v(1,-1) (\lambda_0 - \lambda'\delta) = (\lambda_0 + \lambda'\delta) v^* v(1,-1)$. After some cancellation, this simplifies to $-v^* v (1,-1) \delta = \delta v^* v (1,-1)$. Further substituting our expansion of $v$ and rearranging $\delta$ gives $\delta(b^* b, -a^* a) = \delta(a^* a,-b^* b)$. We thus get that $a^* a = b^* b$. 
  \end{paragraph}
  
  \begin{paragraph}{We seek to show that $a^*a = b^*b = 1/2$:} We know from the previous paragraph that $a^* a = b^* b$. We also know that $\unwind(v) = \begin{bmatrix}a + c\varepsilon \\ b + d\varepsilon \end{bmatrix}$ is a unit vector by the $(\mathbb D, \operatorname{id})$ spectral theorem. So we conclude that $a^* a = b^* b = 1/2$.
  \end{paragraph}

  \begin{paragraph}{We seek to show that $v^* v$ is in $\mathbb D$:} From the previous paragraph, we get that $v^* v = 1/2 + X\delta$ for some $X$. Clearly, $(v^* v)^* = v^* v$, so we conclude $X$ is real.
  \end{paragraph}

  \begin{paragraph}{We prove the main claim:} Now let $w = v (v^* v)^{-1/2}$. Clearly, $w^* w = 1$, as we sought. We also observe that $H w = H v (v^* v)^{-1/2} = v \lambda (v^* v)^{-1/2} = v (v^* v)^{-1/2} \lambda$, with the last equality holding because dual numbers commute. So $Hw = w \lambda$ as desired.
  \end{paragraph}
\end{proof}

\begin{lemma} \label{when-eigenvalues-all-real}
  If all eigenvalues of $\unpack(H)$ (for self-adjoint $\svdalg$ matrix $H$) are real, then there exists an eigenvector $v$ of $H$ for which $v^* v =1$, with eigenvalue of the form $(x,y) \in \mathbb R^2$.
\end{lemma}
\begin{proof}
  Let $\{\unwind(u_1), \unwind(u_2), \dotsc, \unwind(u_{2n})\}$ be an orthonormal eigenbasis of $\unpack(H)$ with corresponding real eigenvalues $\{\lambda_1, \lambda_2, \dotsc, \lambda_{2n}\}$.

  Observe that from $\unpack(H)\unwind(u_i) = \unwind(u_i)\lambda_i$, we get $H u_i = u_i \lambda_i$. Furthermore, we may scale $u_i$ by the scalars $(1,0)$ and $(0,1)$ and the analogous identity remains true: More explicitly, we have that $H u_i(1,0) = u_i (1,0) \lambda_i$ and $H u_i (0,1) = u_i (0,1) \lambda_i$.

  \begin{paragraph}{We seek to show that for each $i$, at least one of $u_i(1,0)$ or $u_i(0,1)$ is not a multiple of $\delta$:} If they are both multiples of $\delta$, then $u_i$ is also a multiple of $\delta$ because $u_i = u_i(1,0) + u_i(0,1)$. But then $\unwind(u_i)$ is a multiple of $\varepsilon$. This is clearly impossible because $\unwind(u_i)$ has unit length.
  \end{paragraph}

  \begin{paragraph}{We seek to show that it is \emph{not} the case that $u_i(1,0)$ is a multiple of $\delta$ for every $i$:} Assume otherwise. We have that the $\unwind$ of a multiple of $\delta$ is a multiple of $\varepsilon$. We thus get that the following is a multiple of $\varepsilon$: $\unwind(u_i(1,0)) = \begin{pmatrix}I_n & 0_n \\ 0_n & 0_n\end{pmatrix} \unwind(u_i)$. This means that the first half of the components of each vector $\unwind(u_i)$ is infinitesimal. But then $\{\unwind(u_i) : i \in \{1,\dotsc,2n\}\}$ cannot be linearly independent.
  \end{paragraph}

  \begin{paragraph}{We seek to show that it is \emph{not} the case that $u_i(0,1)$ is a multiple of $\delta$ for every $i$:} Same argument as above.
  \end{paragraph}

  \begin{paragraph}{We seek to show that there exists a unit eigenvector of $H$:} Pick a $u_i$ such that $u_i(1,0)$ is not a multiple of $\delta$. Either there exists a $u_j$ such that $u_j(0,1)$ is not a multiple of $\delta$ \emph{and} $\lambda_i \neq \lambda_j$, or there doesn't:
    \begin{itemize}
      \item If there doesn't, then we conclude that all eigenvalues of $\unpack(H)$ are the same real number $\lambda$. But then $\unpack(H)$ is a real multiple of the identity matrix. Therefore so is $H$. Any vector is now an eigenvector of $H$, and so we are done.
      \item If there does, then pick this $i$ and $j$. Let $w = u_i(1,0) + u_j(0,1)$. We see that $H w = w(\lambda_i, \lambda_j)$. It remains to show that $w^* w = 1$. But this follows from lemma \ref{semi-eigenvalue-means-eigenvalue}, so we are done. \qedhere
    \end{itemize}
  \end{paragraph}
\end{proof}

The following lemma is necessary to be able to take orthogonal complements. In general, this can be quite complicated if we \emph{only} assume the conditions $P^2 = P$ and $P^* = P$. We therefore need to add the condition that $P$ should be displaced from a real orthogonal projection $Q$ by only a multiple of $\delta$.

\begin{lemma} \label{can-take-orthogonal-complements}
  If $P^2 = P$, $P^* = P$ and $P = Q + \delta (K,K^T)$ for some real projection matrix $Q$ and some real matrix $K$, then $P$ is unitarily diagonalisable with eigenvalues either $0$ or $1$.
\end{lemma}
\begin{proof}
  The real matrix $Q$ can be unitarily diagonalised using a real matrix $U$. We therefore have that $U P U^* = \operatorname{diag}(1,\dotsc,1,0,\dotsc,0)+ \delta (L,L^T)$ for some real matrix $L$. We write $U P U^*$ as a block matrix for clarity: $U P U^* = \begin{pmatrix}I + \delta (L_{11},L_{11}^T) & \delta(L_{12},L_{12}^T)\\ \delta(L_{12},L_{12}^T) & \delta(L_{22}, L_{22}^T)\end{pmatrix}$. If we square this block matrix and recall that $(U P U^*)^2 = U P U^*$, we get that $L_{11}=L_{22}=0$. We can finally diagonalise $U P U^*$ using the block matrix $V = \begin{pmatrix}I & \delta(L_{12}, L_{12}^T) \\ -\delta(L_{12}, L_{12}^T) & I \end{pmatrix}$.
\end{proof}

\begin{theorem} \label{svd-spectral-theorem}
  Every self-adjoint \(\svdalg\)-matrix \(H\) is unitarily diagonalisable with each eigenvalue either a dual number or of the form $(x,y) \in \mathbb R^2$.
\end{theorem}
\begin{proof}
  We prove this by induction on $k$ where $H$ is $k \times k$.

  The claim is clearly true for $k=0$.

  Assume that the claim is true for $k=n-1$. Now $\unpack(H)$ is a self-adjoint $(\mathbb D, \operatorname{id})$ matrix. Therefore, we may obtain a unitarily diagonalisation of it $UDU^* = \unpack(H)$. This gives us a unitary eigenbasis $\{u_1, u_2,\dotsc,u_{2n}\}$ of $\unpack(H)$ with corresponding eigenvalues $\{\lambda_1, \lambda_2, \dotsc, \lambda_{2n}\}$.

  Either all the eigenvalues are real or they're not.

  Consider the case when some eigenvalue $\lambda_i$ is not real, but dual. There is a $u'_i$ such that $\unwind(u'_i) = u_i$. Lemma \ref{dual-eigenvalue-means-eigenvalue} implies that for $w$ obtained by normalising $u'_i$, we have that $H w = w \lambda_i$ and $w^* w = 1$. We have that $\lambda$ in this case is a dual number. By lemma \ref{can-take-orthogonal-complements}, we may take the orthogonal complement of $w$, restrict $H$ to $w^\perp$, and apply the induction hypothesis. We are done.

  Consider the case when all the eigenvalues are real: By lemma \ref{when-eigenvalues-all-real}, we may find a unit eigenvector $v$. $\lambda$ in this case is a pair $(x,y)\in\mathbb R^2$. By lemma \ref{can-take-orthogonal-complements}, we may take the orthogonal complement of $w$. We restrict $H$ to $w^\perp$ and apply the induction hypothesis. We are done.
\end{proof}

\begin{definition}
  We define the spectrum of a self-adjoint \(\svdalg\)-matrix \(H\) (which we will later show is unique) by a triple of finite multisets $(C,L,R)$ where:
  \begin{itemize}
    \item $C$ consists of dual numbers, while $L$ and $R$ consist of real numbers,
    \item $(\forall x \in C) \operatorname{nst}(x) > 0$,
    \item $|L| = |R|$,
    \item $H$ is unitarily similar to $c \oplus (l,r)$ where $c$, $l$ and $r$ are diagonal matrices whose entries belong to $C$, $L$ and $R$ respectively.
  \end{itemize}
\end{definition}

\begin{theorem}
  Given a self-adjoint $\svdalg$ matrix $H$, any two spectra $(C,L,R)$ and $(C',L',R')$ of $H$ are equal.
\end{theorem}
\begin{proof}
  Consider two spectral decompositions of $H$:\begin{itemize}
    \item One where the unitary eigenbasis is $\{u_1,u_2,\dotsc,u_n\}$ with the corresponding spectrum being $(C,L,R)$. We say that the vectors in $\{u_1,\dotsc,u_k\}$ have eigenvalues in $C$.
    \item One where the unitary eigenbasis is $\{u'_1,u'_2,\dotsc,u'_n\}$ with the eigenvalues coming from the spectrum $(C',L',R')$. We say that the vectors in $\{u'_1,\dotsc,u'_{k'}\}$ have eigenvalues in $C'$.
  \end{itemize}
  Observe that $\{u_1, u_1(1,-1), u_2, u_2(1,-1),\dotsc,u_k,u_k(1,-1)\}$  $\cup$ $\{u_{k+1}(1,0),u_{k+1}(0,1)\dotsc,u_n(1,0),u_n(0,1)\}$ form a spanning set. Furthermore, applying $\unwind$ retains the spanning property. We get that for each dual $\lambda_i$, we get that both $\lambda_i$ and $\overline{\lambda_i}$ are eigenvalues of $\unpack(H)$. Since the eigenvalues of $\unpack(H)$ are unique, this establishes that $C = C'$ and $L + R = L' + R'$. By projecting $\operatorname{st}(H)$ on its two components, and applying the uniqueness of real eigenspectra, we get that $L = L'$ and $R = R'$ .
\end{proof}

\begin{corollary}
  $\overline{\operatorname{Herm}}\svdalg$ is a subfree abelian monoid, with the monoid operation being $(C,L,R) + (C',L',R') = (C + C', L + L', R + R')$. This is not free because $|L| = |R|$.
\end{corollary}

\begin{corollary} \label{skew-svd-spectral-theorem}
  A self-adjoint $(\cl_{1,0,1}(\mathbb R), *_{-1})$-matrix is unitarily diagonalisable with a unique spectrum.
\end{corollary}
\begin{proof}
  Proposition \ref{svd-algl-iso} says that $(\cl_{1,0,1}(\mathbb R), *_{-1})$ is isomorphic to $\svdalg$, so this is trivial.
\end{proof}

\section{Proving the spectral theorem for the ``Takagi $*$-algebra'' $\takalg$}

I have discovered this result independently. It is also proved in a paper by Qi et al.\footnote{It's called ``Low Rank Approximation of Dual Complex Matrices'' and remains unpublished. I have not added it to the bibliography because of an ongoing dispute.}
%
%
%
\begin{theorem} \label{tag-spectral-theorem}
  Every self-adjoint \(\takalg\)-matrix \(H\) is unitarily diagonalisable.
\end{theorem}
%



%
\begin{proof}
  Let \(M\) be a self-adjoint $\takalg$-matrix. We find a complex unitary
  matrix \(S\) such that \(S\operatorname{st}(A)S^*\) is diagonal. We let
  \(M'=SMS^*\), which we write as a block matrix \[M' = \begin{pmatrix}
      \lambda_1 I + B_{11} \delta & B_{12} \delta & \dotsb & B_{1n} \delta \\
      B_{12}^T \delta & \lambda_2 I + B_{22} \delta & \ddots & \vdots \\
      \vdots & \ddots & \ddots & B_{n-1,n}\delta \\
      B_{1n}^T \delta & \dotsb & B_{n-1,n}^T \delta & \lambda_n I + B_{nn} \delta
    \end{pmatrix}.\] where each \(B_{ii}\) is complex symmetric. We let
  \[P = \begin{pmatrix}
      I & \frac{B_{12} \delta}{\lambda_1 - \lambda_2} & \dotsb & \frac{B_{1n} \delta}{\lambda_1 - \lambda_n} \\
      -\frac{B_{12}^T \delta}{\lambda_1 - \lambda_2} & I & \ddots & \vdots \\
      \vdots & \ddots & \ddots & \frac{B_{n-1,n} \delta}{\lambda_{n-1} - \lambda_n}\\
      -\frac{B_{1n}^T \delta}{\lambda_1 - \lambda_n} & \dotsb & -\frac{B_{n-1,n}^T \delta}{\lambda_{n-1} - \lambda_n} & I
    \end{pmatrix},\] and let \(M'' = P M' P^*\). We end up with \(M''\)
  being equal to a direct sum of matrices:
  \(M'' = (\lambda_1 I + B_{11} \delta) \oplus (\lambda_2 I + B_{22} \delta) \oplus \dotsb \oplus (\lambda_n I + B_{nn} \delta)\).
  We finally use the Takagi decomposition (whose existence we proved for a general complex-symmetric matrix using the unpack-and-unwind method in proposition \ref{takagi}) to find matrices \(Q_i\) such
  that \(Q_i B_{ii} Q_i^T\) is equal to a real diagonal matrix. We thus
  get that
  \((Q_1 \oplus Q_2 \oplus \dotsb \oplus Q_n) M'' (Q_1 \oplus Q_2 \oplus \dotsb \oplus Q_n)^*\)
  is a diagonal matrix.
\end{proof}

\section{Proving the spectral theorem for the ``skew-Takagi $*$-algebra'' $\skewtakalg$}

Presently, this result has been proved in a paper in Arxiv where I have been promised coauthorship.\footnote{It's called ``Eigenvalues and Singular Value Decomposition of Dual Complex Matrices''. I have not added the paper to the bibliography because of an ongoing dispute.}

\begin{theorem} \label{skew-tag-spectral-theorem}
    Every self-adjoint
    \(\skewtakalg\)-matrix \(H\) is unitarily similar to a direct sum
    of matrices of the form $\begin{pmatrix}\lambda_i & -\lambda'_i \delta \\ \lambda'_i \delta & \lambda_i \end{pmatrix}$ and $(\lambda_i)$ (where $\lambda_i \in \mathbb R$ and $\lambda'_i > 0$).
  \end{theorem}
  
  \begin{proof} Let \(M\) be a self-adjoint $\skewtakalg$-matrix. We find a complex unitary matrix \(S\) such that \(S\operatorname{st}(A)S^*\) is diagonal. We let
  \(M'=SMS^*\), which we write as a block matrix \[M' = \begin{pmatrix}
      \lambda_1 I + B_{11} \delta & B_{12} \delta & \dotsb & B_{1n} \delta \\
      -B_{12}^T \delta & \lambda_2 I + B_{22} \delta & \ddots & \vdots \\
      \vdots & \ddots & \ddots & B_{n-1,n} \delta  \\
      -B_{1n}^T \delta & \dotsb & -B_{n-1,n}^T \delta & \lambda_n I + B_{nn} \delta
    \end{pmatrix}.\] where each \(B_{ii}\) is complex skew-symmetric. We let
  \[P = \begin{pmatrix}
      I & \frac{B_{12} \delta}{\lambda_1 - \lambda_2} & \dotsb & \frac{B_{1n} \delta}{\lambda_1 - \lambda_n} \\
      \frac{B_{12}^T \delta}{\lambda_1 - \lambda_2} & I & \ddots & \vdots \\
      \vdots & \ddots & \ddots & \frac{B_{n-1,n} \delta}{\lambda_{n-1} - \lambda_n}\\
      \frac{B_{1n}^T \delta}{\lambda_1 - \lambda_n} & \dotsb & \frac{B_{n-1,n}^T \delta}{\lambda_{n-1} - \lambda_n} & I
    \end{pmatrix},\] and let \(M'' = P M' P^*\). We end up with \(M''\)
  being equal to a direct sum of matrices:
  \(M'' = (\lambda_1 I + B_{11} \delta) \oplus (\lambda_2 I + B_{22} \delta) \oplus \dotsb \oplus (\lambda_n I + B_{nn} \delta)\).
  We finally use the skew-Takagi decomposition (whose existence is proved in \ref{skew-takagi} using unpack-and-unwind) to find matrices \(Q_i\) such that \(Q_i B_{ii} Q_i^T\) is equal to a direct sum of matrices of the form $\begin{pmatrix}0 & -\lambda'_i \\ \lambda'_i & 0 \end{pmatrix}$ and $(0)$ (with the last type of block only occurring once). We thus get that
  \((Q_1 \oplus Q_2 \oplus \dotsb \oplus Q_n) M'' (Q_1 \oplus Q_2 \oplus \dotsb \oplus Q_n)^*\) is of the required form.
  \end{proof}

\section{Spectral theorems for $\cl_{p,q,0}(\mathbb R)$ (for 1 involution) and $\cl_{p,q,1}(\mathbb R)$ (for 2 involutions)}

The results here follow easily from those of the previous sections, and the classification theorem for Clifford algebras over $\mathbb R$.

\newcommand{\clpq}{(\cl_{p,q,0}(\mathbb R), *)}
\newcommand{\clpqonetriv}{(\cl_{p,q,1}(\mathbb R), *_1)}
\newcommand{\clpqonenontriv}{(\cl_{p,q,1}(\mathbb R), *_{-1})}
\begin{definition}
  By $\clpq$, we mean the $*$-algebra whose underlying algebra is (of course) $\cl_{p,q,0}(\mathbb R)$ and whose involution is defined (uniquely) to satisfy $z^* = -z$ if $z^2 = -1$ and $z^* = z$ if $z^2 = 1$.
\end{definition}
\begin{definition}
  By $\clpqonetriv$, we mean the $*$-algebra whose underlying algebra is (of course) $\cl_{p,q,1}(\mathbb R)$ and whose involution is defined the same as in $\clpq$, but with $\delta^{*_1} = \delta$.
\end{definition}
\begin{definition}
  By $\clpqonenontriv$, we mean the $*$-algebra whose underlying algebra is (of course) $\cl_{p,q,1}(\mathbb R)$ and whose involution is defined the same as in $\clpq$, but with $\delta^{*_{-1}} = -\delta$.
\end{definition}

The following has been proven elsewhere \cite{hile_matrix_1990}, and we recall it.

\begin{lemma} \label{classification-clpq}
  Every $\clpq$ is isomorphic to one of the following $*$-algebras:
  \begin{enumerate}
    \item \(\mathbb R\) understood to have the trivial involution,
    \item \(\mathbb C\) understood to be equipped with complex-conjugation as its involution (which coincides with the transpose of its $\mathbb R$-matrix representation),
    \item \(\mathbb H\) understood to be equipped with its standard involution (which coincides with the conjugate-transpose of its $\mathbb R$-matrix and $\mathbb C$-matrix representations),
    \item \(\mathbb R \oplus \mathbb R\) understood as a direct sum of $*$-algebras,
    \item \(\mathbb C \oplus \mathbb C\) understood as a direct sum of $*$-algebras,
    \item \(\mathbb H \oplus \mathbb H\) understood as a direct sum of $*$-algebras,
    \item \(M_{2^n}(R)\) understood as having the conjugate-transpose as its involution, with the ``conjugate'' (more correctly, involution) being inherited from the $*$-algebra $R$,
  \end{enumerate}
  for some values of \(n\), and where \(R\) is any of the $*$-algebras above. This is a bit stronger than the usual statement of the classification theorem for $\cl_{p,q,0}(\mathbb R)$ in that it's an isomorphism of $*$-algebras and not just algebras.
\end{lemma}

We can now prove:

\begin{proposition}Every self-adjoint matrix \(H\) over
    \(\clpq\) admits a $\clpq$-unitary
    matrix \(U\) and a $\clpq$-diagonal real matrix \(D\)
    such that \(UDU^* = H\).
\end{proposition}
    
\begin{proof}    
  Using lemma \ref{classification-clpq}, we prove our theorem for each of the seven cases within that lemma in turn:
  
  \begin{enumerate}
  \def\labelenumi{\arabic{enumi}.}
  \tightlist
  \item
    First, \(H\) has a real eigenvector \(v\) with real eigenvalue. This
    follows either from the Fundamental Theorem of Algebra or from
    Lagrange multipliers (we won't show the details here because they're
    well covered elsewhere). Finally, restrict \(H\) to the orthogonal
    complement of \(v\). \(H\) over this orthogonal complement is still
    self-adjoint, so the construction can be repeated.
  \item
    We use the unpack-and-unwind method. Let $\unpack_{\mathbb C}(A + Bi) := \begin{pmatrix}A & -B \\ B & A\end{pmatrix}$ for $A$ and $B$ any $\mathbb R$-matrices. Observe that \(\unpack_{\mathbb C}(H)\) is self-adjoint over \(\mathbb R\), and hence has a real
    eigenvector \(v\) of real eigenvalue \(\lambda\). Let $\unwind_{\mathbb C}(w) := \begin{pmatrix} \Re(w) \\ \Im(w)\end{pmatrix}$. We have that $v = \unwind(v')$ for some $\mathbb C$-vector $v'$. We thus have that $\unwind(H v') = \unpack(H) \unwind(v') = \lambda \unwind(v') = \unwind(\lambda v')$. Since $\unwind$ is injective, we may cancel to get $Hv' = \lambda v'$. We restrict $H$ to the orthogonal complement of $v'$, and repeat.
  \item
  We use the unpack-and-unwind method. Let $\unpack_{\mathbb H}(A + Bi + Cj + Dk) = \unpack_{\mathbb C}\begin{pmatrix}A - Bi & -C + Di\\ C + Di & A + Bi \end{pmatrix}$. Observe that \(\unpack_{\mathbb H}(H)\) is self-adjoint over \(\mathbb R\), and hence has a real
    eigenvector \(v\) of real eigenvalue \(\lambda\).  Let $\unwind_{\mathbb H}(ai + bi + cj + dk) := \begin{pmatrix} a \\ b \\ c \\ d\end{pmatrix}$, where $a, b, c, d$ are $\mathbb R$-vectors. We have that $v = \unwind(v')$ for some $\mathbb H$-vector $v'$. We thus have that $\unwind(H v') = \unpack(H) \unwind(v') = \lambda \unwind(v') = \unwind(\lambda v')$. We restruct $H$ to the orthogonal complement of $v'$, and repeat.
  \item
    If \(M\) is a self-adjoint over \(\mathbb R \oplus \mathbb R\)
    then
    \(M = (L,K)\) where \(L\) and \(K\) are self-adjoint matrices over
    \(\mathbb R\). Then this reduces to case 1.
  \item
    If \(M\) is a self-adjoint over \(\mathbb C \oplus \mathbb C\)
    then
    \(M = (L,K)\) where \(L\) and \(K\) are self-adjoint matrices over
    \(\mathbb C\). Then this reduces to case 2.
  \item
    If \(M\) is a self-adjoint over \(\mathbb H \oplus \mathbb H\)
    then
    \(M = (L,K)\) where \(L\) and \(K\) are self-adjoint matrices over
    \(\mathbb H\). Then this reduces to case 3.
  \item
    By block matrices, this reduces to the first six cases.\qedhere
  \end{enumerate}
\end{proof}

\begin{proposition}Every self-adjoint matrix \(H\) over
    \(\clpqonetriv\) admits a $\clpqonetriv$-matrix \(U\) and a diagonal $\clpqonetriv$-matrix \(D\) such that \(UDU^* = H\).
\end{proposition}
    
\begin{proof}
  We add an element $\delta$ to each of the seven cases in lemma \ref{classification-clpq}. This element $\delta$ satisfies:
  
  \begin{enumerate}
  \def\labelenumi{\arabic{enumi}.}
  \tightlist
  \item
    \(\delta\) is in the centre.
  \item
    \(\delta i = -i \delta\)
  \item
    \(\delta\) is in the centre.
  \item
    \(\delta (x,y) = (y,x) \delta\).
  \item
    \(\delta\) is in the centre.
  \item
    \(\delta\) is in the centre.
  \item
    The algebra is \(M_{2^n}(R \cup \{\delta\})\).
  \end{enumerate}
  
  In each case, the involution is extended so that $\delta^* = \delta$. We now prove the theorem for each of the seven cases in turn:
  
  \begin{enumerate}
  \def\labelenumi{\arabic{enumi}.}
  \tightlist
  \item
    This has been proven before in my paper on the dual number spectral
    theorem. \cite{gutin_generalizations_2021}
  \item
    See theorem \ref{tag-spectral-theorem}.
  \item We use the unpack-and-unwind method. Let $\unpack_{\mathbb H \otimes \mathbb D}(A + Bi + Cj + Dk) = \unpack_{\mathbb C \otimes \mathbb D}\begin{pmatrix}A - Bi & -C + Di\\ C + Di & A + Bi \end{pmatrix}$ (where $A, B, C, D$ are $\mathbb D$-matrices). Observe that \(\unpack(H)\) is self-adjoint over \(\mathbb D\), and hence has a \(\mathbb D\)-eigenvector \(v\) of \(\mathbb D\)-eigenvalue
    \(\lambda\) (a fact which follows from the dual-number spectral theorem \cite{gutin_generalizations_2021}). Let $\unwind_{\mathbb H \otimes \mathbb D}(ai + bi + cj + dk) := \begin{pmatrix} a \\ b \\ c \\ d\end{pmatrix}$, where $a, b, c, d$ are $\mathbb D$-vectors. We have that $v = \unwind(v')$ for some $\mathbb H \otimes \mathbb D$-vector $v'$. We thus have that $\unwind(H v') = \unpack(H) \unwind(v') = \lambda \unwind(v') = \unwind(\lambda v')$. We restruct $H$ to the orthogonal complement of $v'$, and repeat.
  \item
    See theorem \ref{svd-spectral-theorem}.
  \item
    If \(M\) is a self-adjoint over
    \((\mathbb C \otimes \mathbb D) \oplus (\mathbb C \otimes \mathbb D)\)
    then
    \(M = (L,K)\) where \(L\) and \(K\) are self-adjoint matrices over
    \(\mathbb C \otimes \mathbb D\). Then this reduces to case 2.
  \item
    If \(M\) is a self-adjoint over
    \((\mathbb H \otimes \mathbb D)\oplus (\mathbb H \otimes \mathbb D)\)
    then
    \(M = (L,K)\) where \(L\) and \(K\) are self-adjoint matrices over
    \(\mathbb H\otimes \mathbb D\). Then this reduces to case 3.
  \item
    By block matrices, this reduces to the first six cases.\qedhere
  \end{enumerate}
\end{proof}

\newcommand{\dualquaternionnontriv}{(\cl_{1,1,1}(\mathbb R),*_{-1})}
\begin{lemma} \label{skew-quaternion-generalised-spectral-theorem}
    Every self-adjoint
    \(\dualquaternionnontriv\)-matrix \(H\) is unitarily similar to a direct sum
    of matrices of the form $\begin{pmatrix}\lambda_i & -\lambda'_i \delta \\ \lambda'_i \delta & \lambda_i \end{pmatrix}$ and $(\lambda_i)$ (where $\lambda_i \in \mathbb R$ and $\lambda'_i > 0$).
\end{lemma}
  
  \begin{proof} Let \(M\) be a self-adjoint $\dualquaternionnontriv$-matrix. We find a quaternion unitary matrix \(S\) such that \(S\operatorname{st}(A)S^*\) is diagonal. We let
  \(M'=SMS^*\), which we write as a block matrix \[M' = \begin{pmatrix}
      \lambda_1 I + B_{11} \delta & B_{12} \delta & \dotsb & B_{1n} \delta \\
      -B_{12}^T \delta & \lambda_2 I + B_{22} \delta & \ddots & \vdots \\
      \vdots & \ddots & \ddots & B_{n-1,n} \delta  \\
      -B_{1n}^T \delta & \dotsb & -B_{n-1,n}^T \delta & \lambda_n I + B_{nn} \delta
    \end{pmatrix}.\] where each \(B_{ii}\) is complex skew-Hermitian. We let
  \[P = \begin{pmatrix}
      I & \frac{B_{12} \delta}{\lambda_1 - \lambda_2} & \dotsb & \frac{B_{1n} \delta}{\lambda_1 - \lambda_n} \\
      \frac{B_{12}^* \delta}{\lambda_1 - \lambda_2} & I & \ddots & \vdots \\
      \vdots & \ddots & \ddots & \frac{B_{n-1,n} \delta}{\lambda_{n-1} - \lambda_n}\\
      \frac{B_{1n}^* \delta}{\lambda_1 - \lambda_n} & \dotsb & \frac{B_{n-1,n}^* \delta}{\lambda_{n-1} - \lambda_n} & I
    \end{pmatrix},\] and let \(M'' = P M' P^*\). We end up with \(M''\)
  being equal to a direct sum of matrices:
  \(M'' = (\lambda_1 I + B_{11} \delta) \oplus (\lambda_2 I + B_{22} \delta) \oplus \dotsb \oplus (\lambda_n I + B_{nn} \delta)\).
  We finally use proposition \ref{skew-quaternion-spectral} to find matrices \(Q_i\) such that \(Q_i B_{ii} Q_i^*\) is equal to a direct sum of matrices of the form $\begin{pmatrix}0 & -\lambda'_i \\ \lambda'_i & 0 \end{pmatrix}$ and $(0)$. We thus get that
  \((Q_1 \oplus Q_2 \oplus \dotsb \oplus Q_n) M'' (Q_1 \oplus Q_2 \oplus \dotsb \oplus Q_n)^*\) is of the required form.
  \end{proof}

\begin{proposition}
  Every self-adjoint matrix \(H\) over
  \(\clpqonenontriv\) admits a unitary
  \(\clpqonenontriv\)-matrix \(U\) such that:
  \begin{itemize}
      \item when $p = q = 0$ or $q = 1$ and $p = 0$, there is a block-diagonal $D$ such \(UDU^* = H\) where $D$'s blocks are of the form $\begin{pmatrix}\lambda_i & -\lambda'_i \delta \\ \lambda'_i \delta & \lambda_i \end{pmatrix}$, $(\lambda_i)$, $\begin{pmatrix}\lambda'_i \delta \end{pmatrix}$ or $(0)$, where $\lambda_i, \lambda'_i$ are real and strictly positive.
      \item when $p = 1$ and $q = 0$, there is a block-diagonal $D$ such \(UDU^* = H\) where $D$'s blocks are of the form $\begin{pmatrix}\lambda_i & -\lambda'_i \delta \\ \lambda'_i \delta & \lambda_i \end{pmatrix}$, $(\lambda_i)$, $\begin{pmatrix}\lambda'_i \delta \end{pmatrix}$ or $(0)$, where the $\lambda_i$ are in $\mathbb R \oplus \mathbb R$ and the $\lambda'_i$ are real and strictly positive.
      \item when $p + q > 1$, $H = UDU^*$ where $D$ is \(\clpqonenontriv\)-diagonal.
  \end{itemize}
\end{proposition}

\begin{proof}
  We add an element $\delta$ to each of the seven cases in lemma \ref{classification-clpq}. This element $\delta$ satisfies:
  
  \begin{enumerate}
  \def\labelenumi{\arabic{enumi}.}
  \tightlist
  \item
    \(\delta\) is in the centre.
  \item
    \(\delta i = -i \delta\)
  \item
    \(\delta\) is in the centre.
  \item
    \(\delta (x,y) = (y,x) \delta\).
  \item
    \(\delta\) is in the centre.
  \item
    \(\delta\) is in the centre.
  \item
    The algebra is \(M_{2^n}(R \cup \{\delta\})\).
  \end{enumerate}
  
  This exhausts all Clifford algebras of the form
  \(Cl_{p,q,1}(\mathbb R)\). Using this, we prove our theorem for each of
  the seven cases in turn:
  
  \begin{enumerate}
  \def\labelenumi{\arabic{enumi}.}
  \tightlist
  \item
    This has been proven before in my paper on the dual number spectral
    theorem. \cite{gutin_generalizations_2021}
  \item
    See theorem \ref{skew-tag-spectral-theorem}.
  \item
    See lemma \ref{skew-quaternion-generalised-spectral-theorem}.
  \item
    See corollary \ref{skew-svd-spectral-theorem}.
  \item
    If \(M\) is a self-adjoint over
    \((\mathbb C \otimes \mathbb D) \oplus (\mathbb C \otimes \mathbb D)\)
    then
    \(M = (L,K)\) where \(L\) and \(K\) are self-adjoint matrices over
    \(\mathbb C \otimes \mathbb D\). Then this reduces to case 2.
  \item
    If \(M\) is a self-adjoint over
    \((\mathbb H \otimes \mathbb D)\oplus (\mathbb H \otimes \mathbb D)\)
    then
    \(M = (L,K)\) where \(L\) and \(K\) are self-adjoint matrices over
    \(\mathbb H\otimes \mathbb D\). Then this reduces to case 3.
  \item
    By block matrices, this reduces to the first six cases.\qedhere
  \end{enumerate}
\end{proof}

\bibliographystyle{plain}
\bibliography{paper}

\end{document}